\documentclass[a4paper,12pt]{amsart}
\usepackage[english]{babel}
\usepackage[utf8]{inputenc}
\usepackage[T1]{fontenc}
\usepackage[style=alphabetic,useprefix,hyperref,backend=bibtex,maxbibnames=99,%
	firstinits=true
	]{biblatex}
\bibliography{./BIB-RegularityOfSpheres-2015-05-11}

\usepackage{amssymb}
\usepackage{mathrsfs}
\usepackage{hyperref}
\usepackage[usenames,dvipsnames]{xcolor}
\hypersetup{colorlinks,%
citecolor=OliveGreen,%
filecolor=magenta,%
linkcolor=RoyalBlue,%
urlcolor=cyan}
\usepackage{csquotes}
\usepackage{mparhack}	
\usepackage{comment}	
\usepackage{marvosym}	
\usepackage{enumitem}	
\usepackage{array}		
\usepackage{todonotes}

\newcommand{\scr}[1]{\mathscr{#1}}
\newcommand{\frk}[1]{\mathfrak{#1}}
\newcommand{\bb}[1]{\mathbb{#1}}

\renewcommand{\rm}[1]{\mathrm{#1}}

\newcommand{\N}{\mathbb{N}}	
\newcommand{\R}{\mathbb{R}}	

\newcommand{\HH}{\mathbb{H}}	
\newcommand{\E}{\mathbb{E}}	
\newcommand{\Co}{\mathscr{C}}	

\renewcommand{\Vec}{\mathrm{Vec}}	
\newcommand{\Span}{\mathrm{span}}	

\newcommand{\ssubset}{\subset\subset}	
\newcommand{\dd}{\,\mathrm{d}}	
\newcommand{\de}{\partial}		
\newcommand{\LL}{\mathtt{L}}	
\newcommand{\LI}{\mathtt{L^{\infty}}}	
\newcommand{\THEN}{\Rightarrow}	
\newcommand{\IFF}{\Leftrightarrow}	

\newcommand{\ddx}{\framebox[\width]{$\Rightarrow$} }
\newcommand{\ssx}{\framebox[\width]{$\Leftarrow$} }

\renewcommand{\Im}{\mathrm{Im}}

\theoremstyle{plain}
\newtheorem{Prop}{Proposition}[section]
\newtheorem{Teo}[Prop]{Theorem}
	\newtheorem*{Teo*}{Theorem}
\newtheorem{Lem}[Prop]{Lemma}

\newtheorem{Cor}[Prop]{Corollary}

\theoremstyle{definition}
\newtheorem{Def}[Prop]{Definition}
\newtheorem{Rem}[Prop]{Remark}

\newtheoremstyle{esempio}
	{3pt}
	{3pt}
	{}
	{}
	{}
	{:}
	{.5em}
	{}
  
\theoremstyle{esempio}
\newtheorem{ese}[Prop]{\textsc{Example}}

\newtheorem{exe}[Prop]{\textsc{Exercise}}

\theoremstyle{remark}

\usepackage{dsfont}

\newcommand{\dcc}{d_{CC}}

\newcommand{\diam}{\rm{diam}}

\newcommand{\f}{\mathtt{f}}

\newcommand{\LUI}{\LI([0,1];\E)}

\newcommand{\End}{\mathtt{End}}

\DeclareMathOperator*{\esup}{ess\,sup}

\newcommand{\wsto}{\overset*{\rightharpoonup}} 

\renewcommand{\HH}{\bb H}	
\newcommand{\h}{\frk h}
\newcommand{\g}{\frk g} 	

\newcommand{\str}{\tau} 
\newcommand{\Cone}{\mathtt{Cone}}
\newcommand{\Ad}{\rm{Ad}}
\newcommand{\ad}{\rm{ad}}
\newcommand{\Int}{\rm{int}}

\title[Regularity of spheres]{Regularity properties of spheres in homogeneous groups}
\author[Le Donne]{Enrico Le Donne}
\author[Nicolussi Golo]{Sebastiano Nicolussi Golo}

\thanks{S.N.G.~has been supported by the People Programme (Marie Curie Actions) of the European Union’s Seventh Framework Programme FP7/2007-2013/ under REA grant agreement n. 607643.}
\address[Le Donne, Nicolussi Golo]{University of Jyvaskyla, Finland}%

\date{\today}
\subjclass[2010]{%
28A75, 
22E25,  
53C60, 
53C17, 
26A16, 
}
\keywords{sub-Finsler geometry, %
sub-Riemannian geometry, %
Carnot groups, %
Lipschitz regularity, %
abnormal geodesic, %
Heisenberg group%
}

\begin{document}
\maketitle

\begin{abstract} 
	We study left-invariant distances on Lie groups for which
	there exists a one-parameter family of homothetic automorphisms.
	The main examples are Carnot groups,
	in particular the Heisenberg group with the standard dilations.	
	We are interested in criteria implying 
	that, locally and away from the diagonal, the distance is Euclidean Lipschitz 
	and, consequently,
	that the metric spheres are boundaries of Lipschitz domains in the Euclidean sense.
	In the first part of the paper, we consider geodesic distances. 
	In this case, we actually prove the regularity of the distance in the more general context of sub-Finsler manifolds with no abnormal geodesics.
	Secondly, for general groups we identify an algebraic criterium
	in terms of the dilating automorphisms, which for example makes us
	conclude the regularity of homogeneous distances 
	on the Heisenberg group.
	In such a group, we analyze in more details the geometry of metric spheres.
	We also provide examples
	of homogeneous groups
	 where spheres presents cusps.
\end{abstract}

\phantomsection
\addcontentsline{toc}{section}{Contents}
\setcounter{tocdepth}{2}
\tableofcontents

\section{Introduction}

The study of the asymptotic geometry of groups lead us to investigate   spheres in homogeneous groups, examples of which are asymptotic cones of  finitely generated nilpotent groups. 
%
A homogeneous group is a Lie group $G$ endowed with a family of Lie group automorphisms $\{\delta_\lambda\}_{\lambda>0}$ and a left-invariant distance $d$ for which each $\delta_\lambda$ multiplies the distance by $\lambda$, 
 see Section~\ref{subs08291614}.
%
An algebraic characterization of these groups is known  by \cite{MR812604}.
In fact, 
the Lie algebra $\g$ of $G$ admits a \emph{grading},
i.e., a decomposition $\g=\bigoplus_{i\geq1} V_i$ such that $[V_i,V_j]\subset V_{i+j}$.
For simplicity, we assume that
 the dilations are 
the ones induced by the grading. Namely, the 
dilation of factor $\lambda$ {\em relative} to the grading
 is the one
such that
$(\delta_\lambda)_*(v)=\lambda^i v$ for all $v\in V_i$.
Not all homogeneous metrics admit this type of dilations.
We denote by $0$ the neutral element of   $G$ and by $\bb S_d$ the unit sphere at 0 for a distance $d$ on $G$, i.e., $\bb S_d:=\{p\in G:d(0,p)=1\}$.

In this paper we want to exclude  cusps in spheres
since their presence in the asymptotic cone 
of a finitely generated nilpotent group 
may give a 
slower
 rate of convergence in the blow down, see \cite{MR3153949}.
We
find criteria 
implying that the metric spheres are boundaries of Lipschitz domains
and in fact that 
the distance function from a point is a locally Lipschitz function with respect to a Riemannian metric.
\\

First, we address the case where the distance $d$ is a length distance. 
Thanks to a characterization of Carnot groups, see \cite{MR3283670}, the group $G$ is in this case a stratified group and $d$ is a sub-Finsler distance.
Being a \emph{stratified group} means that the grading of $\g$ 
is such that the first layer $V_1$ generates $\g$. 
Being a \emph{sub-Finsler distance} means that there are a left-invariant subbundle $\Delta\subset TG$ and a left-invariant norm $\|\cdot\|$ on $\Delta$ such that the length induced by $d$ of an absolutely continuous curve $\gamma:[0,1]\to G$ is equal to 
$
\int_0^1 \|\gamma'(t)\| \dd t ,
$
where $\|\gamma'(t)\|=+\infty$ if $\gamma'(t)\notin\Delta$.
The left-invariant subbundle $\Delta$ is actually the one generated by $V_1$.

In the sub-Finsler case, 
an obstruction to Lipschitz regularity of the sphere comes from the presence of 
length-minimizing curves (also called geodesics) that are not regular, in the sense that
the first variation parallel to the subbundle $\Delta$ does not have maximal rank, see Definition~\ref{def07101125}.


\begin{Teo}\label{teo04221818}
	Let $G$ be a stratified group endowed with a sub-Finsler metric $d$.
	Let $d_0:G\to[0,+\infty)$, $p\mapsto d(0,p)$.
	Let $p\in G$ be such that all geodesics from $0$ to $p$ are regular.
	Then, for any Riemannian metric $\rho$ on $G$ the function $d_0$ is Lipschitz with respect to $\rho$ in some neighborhood of $p$.
\end{Teo}
We will actually state and prove Theorem~\ref{teo04221818} in the more general setting of \emph{sub-Finsler manifolds of  constant-type norm}, see Section~\ref{subs07101127}.

 In   case of homogeneity, the regularity of  the distance implies also the regularity of the spheres.
Hence, 
using Theorem~\ref{teo04221818}
we easily get the second   result
for sub-Finsler homogeneous groups.
\begin{Teo}\label{teo04221825}
	Let $G$ be a stratified group endowed with a sub-Finsler metric $d$.
	Let $p\in\bb S_d$ be such that all geodesics from $0$ to $p$ are regular.
	Then, in smooth coordinates,
the set $\bb S_d$ is a Lipschitz graph
in some neighborhood of $p$.
	In particular, if all non-constant geodesics are regular, then metric balls 
	are 
	Lipschitz domains.
\end{Teo}
   Notice that 
a ball may be a Lipschitz domain even if the distance from a point is not Lipschitz (we give an example in Remark \ref{rem09051840}).
In Section~\ref{sec07101128} we also present examples of sub-Riemannian and sub-Finsler distances whose balls have a cusp.
\\

At a second stage, we drop the hypothesis of $d$ being a length distance and we present a result similar to the previous Theorem~\ref{teo04221825} in the context of homogeneous groups.
Hereafter we denote by $L_p$ and $R_p$ the left and the right translations on $G$, respectively, and by $\bar\delta(p)$ the vector $\left.\frac{\dd}{\dd t} \delta_t(p)\right|_{t=1}\in T_pG$, where  $\{\delta_t\}_{t>0}$ are the dilations relative to a grading.
\begin{Teo}\label{teo04221835}
	Let $(G,d)$ be a homogeneous group with dilations relative to a grading, see Definition~\ref{def08291617}.
	Assume $p\in\bb S_d$ is such that 
	\begin{equation}\label{eq04221835}
		\dd L_p(V_1) + \dd R_p(V_1) + \Span\{\bar \delta(p)\} = T_pG .
	\end{equation}	
	Then, in some neighborhood of $p$ we have that 
	the sphere $\bb S_d$ is a Lipschitz graph    
	and the distance $d_0$ from the identity is Lipschitz with respect to any Riemannian metric $\rho$.
\end{Teo} 

The similarity between Theorem~\ref{teo04221825} and Theorem~\ref{teo04221835} consists in the fact that, if $d$ is a sub-Finsler distance, then condition \eqref{eq04221835} implies that all geodesics from $0$ to $p$ are regular, see Remark~\ref{rem07101129}.\\

The equality \eqref{eq04221835} or the absence of  non-regular  geodesics are actually quite strong conditions.
However, in general we can give an upper bound for the Hausdorff dimension of spheres.
In fact, if $d$ is a homogeneous distance on a graded group of maximal degree $s$, then
\begin{equation}\label{eq04282051}
 	\dim_H^\rho(G) -1 \le \dim_H^\rho(\bb S_d) \le \dim_H^\rho(G)-\frac1s ,
\end{equation}
where $\dim_H^\rho$ is the Hausdorff dimension with respect to some (therefore any) Riemannian metric $\rho$.
We show with Proposition~\ref{prop05281759} that this estimate is sharp.
\\

In the last part of the paper, we analyze in more details an important specific example: the Heisenberg group.
In this graded group we consider all possible homogeneous distances and prove that in exponential coordinates
\begin{enumerate}[label=(\roman*)]
\item 	the unit ball is a star-shaped Lipschitz domain (Proposition~\ref{prop05291014});
\item 	the unit sphere is a locally Lipschitz graph with respect to the direction of the center of the group (Proposition~\ref{prop05291027}).
\end{enumerate}
We also give a method to construct homogeneous distances in the Heisenberg group with arbitrary Lipschitz regularity of the sphere.
Namely, the graph of each Lipschitz function defined on the unit disk, up to adding to it a constant, is the sphere of some homogeneous distance, see Proposition~\ref{prop05291028}.
The investigation of this class of examples is meaningful in connection to Besicovitch's covering property as studied in \cite{2014arXiv1406.1484L} and \cite{LeDonne_Rigot_rmknobcp}. 
\\

The paper is organized as follows.
In section~\ref{sec08291629} we will present all preliminary notions needed in the paper.
We introduce sub-Finsler manifolds of constant-type norm, graded and homogeneous groups and Carnot groups.
Section~\ref{sec08291633} is devoted to the proof of Theorem~\ref{teo04221818};
first in the setting of sub-Finsler manifolds, see Theorem~\ref{teo04241155} proved in Section~\ref{subs07122224}, 
then with a
 more specific result for Carnot groups, see Proposition~\ref{prop06031443}.
 In Section~\ref{sec05111258} we
see    metric spheres as  graphs over  smooth spheres.
Hence, we show Theorem~\ref{teo04221825},   the inequalities \eqref{eq04282051}, and Theorem~\ref{teo04221835}.
In Section~\ref{sec07101128} we present six examples: three different grading of $\R^2$, the Heisenberg group, a sub-Finsler sphere with a cusp and a sub-Riemannian sphere with a cusp.
In Section~\ref{sec05261301} we prove stronger properties for spheres of homogeneous distances on the Heisenberg group.

\section{Preliminaries}\label{sec08291629}
\subsection{Sub-Finsler structures}\label{subs07101127}
Let $M$ be a manifold of dimension $n$.
We will write $TM$ for the tangent bundle and $\Vec(M)$ for the space of smooth vector fields on $M$.
\begin{Def}[Sub-Finsler structure]
	A \emph{sub-Finsler structure (of constant-type norm) of rank $r$} on a manifold $M$ is a triple $(\E, \|\cdot\|, \f)$, where $(\E,\|\cdot\|)$ is a normed vector space of dimension $r$ and $\f:M\times\E\to TM$ is a smooth bundle morphism with $\f(\{p\}\times\E) \subset T_pM$, for all $p\in M$.
\end{Def}

\begin{Def}[Horizontal curve]
	A curve $\gamma:[0,1]\to M$ is a \emph{horizontal curve} if it is absolutely continuous and there is $u:[0,1]\to\E$ measurable, which is called \emph{a control of $\gamma$}, such that 
	\[
	\gamma'(t) = \f(\gamma(t), u(t))
	\qquad\text{for a.e. } t\in[0,1].
	\]
	In this case $\gamma$ is called \emph{integral curve} of $u$ and we write $\gamma_u$.
\end{Def}

\begin{Def}[Space of controls]
	The space of \emph{$\LI$-controls} is defined as\footnote{
	Among the three norms $\LL^1$, $\LL^2$ and $\LL^\infty$ for controls, we chose the latter because 
	the unit ball in $\LL^1([0,1];\E)$ is not weakly compact
	and the $\LL^2$-space is not a Hilbert space in our contest.
	}
	\[
	\LI([0,1];\E) := \left\{ u:[0,1]\to \E\text{ measurable, } \esup_{t\in[0,1]}\|u(t)\| < \infty \right\} .
	\]
	This is a Banach space with norm $\|u\|_{\LI} := \esup_{t\in[0,1]}\|u(t)\|$.
\end{Def}

Thanks to known results for ordinary differential equations, see \cite{MR3308395},
given a control $u\in\LUI$, there is 
an open neighborhood $M^*\subset M\times[0,1]$ of $M\times\{0\}$
and a map $\Phi^u:M^*\to M$,
such that for all $p\in M$ the set $M^*\cap(\{p\}\times[0,1])$ is connected and contains $(p,0)$, and the curve $\gamma(t):=\Phi^u(p,t)$ is the unique solution to the Cauchy problem
\[
\begin{cases}
 	\gamma(0)=p \\
	\gamma'(t) = \f(\gamma(t),u(t)) &\text{for a.e.}~t\text{ with }(p,t)\in M^*.
\end{cases}
\]
The map $\Phi^u$ is called \emph{flow} of $u$.

\begin{Rem}\label{rem06011524}
	We will always assume that every $u\in\LUI$ has a flow defined on the entire interval $[0,1]$, i.e., $M^*=M\times[0,1]$.
	This happens in many cases, for example for left-invariant sub-Finsler structures on Lie groups, in particular in Carnot groups.
\end{Rem}

\begin{Def}[End-point map]
	Fix $o\in M$.
	Define 
	the \emph{End-point map} with \emph{basis point} $o$, $\End_o:\LUI\to M$, as
	\[
	\End_o(u) = \Phi^u(o,1) .
	\]
\end{Def}

By standard result of ODE the map $\End_o$ is of class $\Co^1$, see \cite{MR3308395}.

\begin{Def}[Regular curves]\label{def07101125}
	Given $o\in M$, a control $u\in\LUI$ is said to be 
	\emph{regular} if it is a regular point of $\End_o$, i.e., if 
	$\dd\End_o(u):\LUI\to T_{\End_o(u)}M$ is surjective.
	A \emph{singular control} is a control that is not regular.
\end{Def}

\begin{Def}[Sub-Finsler distance]
	The \emph{sub-Finsler distance}, also called \emph{Carnot-Carathéodory distance}, 
	between two points $p,q\in M$ is 
	\[
	d(p,q) := \inf\left\{\int_0^1\|u(t)\|\dd t:u\in\LUI\text{ with }\End_p(u)=q\right\} .
	\]
\end{Def}
Clearly $(M,d)$ is a metric space, even though it might happen $d(p,q)=\infty$.
Let $\ell_d(\gamma)$ be the     length of a curve $\gamma$ with respect to $d$, see \cite{MR2039660}.
It can be proven that a  curve $\gamma:[0,1]\to (M,d)$ is Lipschitz if and only if it is horizontal and it admits a control in $\LUI$. 
Moreover, if $\gamma$ is Lipschitz, then
\[
\ell_d(\gamma) = \inf\left\{ \int_0^1\|u(t)\|\dd t : u\in\LUI\text{ control of $\gamma$} \right\}.
\]
We will use the term \emph{geodesic} as a synonym of \emph{length-minimizer}.

The distance can be expressed by using the $\LI$-norm as an energy, i.e., 
for every $p,q\in M$
\[
d(p,q) = \inf\left\{\|u\|_{\LI}:u\in\LUI\text{ with }\End_p(u)=q\right\} .
\]
Moreover, if $u$ realizes the infimum above, then its integral curve $\gamma_u$ starting from $p$ is a length-minimizing curve parametrized by constant velocity, i.e., 
\[
d(p,q) = \|u\|_{\LI} = \ell_d(\gamma_u) = \|u(t)\| ,
\qquad\text{for a.e.}t\in[0,1] .
\]

\begin{Def}[Bracket-generating condition]
	Let $\scr A$ be the Lie algebra generated by the set 
	\[
	\left\{p\mapsto \f(p,X(p)): X:M\to\E\text{ smooth}\right\} \subset \Vec(M) .
	\]
	We say that the sub-Finsler structure $(\E,\|\cdot\|,\f)$ on $M$ satisfies the \emph{bracket-generating condition} if for all $p\in M$
	\[
	\{V(p):V\in \scr A\} = T_pM .
	\]
\end{Def}

As a consequence of the Orbit Theorem \cite{MR1425878}, we have the following basic well-known fact.
\begin{Lem}
	If $(\E,\|\cdot\|,\f)$ satisfies the bracket-generating condition, then the distance $d$ induces the original topology of $M$ and $(M,d)$ is a locally compact and locally geodesic length space.
\end{Lem}
By the Hopf-Rinow Theorem, see \cite{Bur01A-course-in-metric-geometry}, the assumption in Remark~\ref{rem06011524} implies that $(M,d)$ is a complete, boundedly compact metric space.

\subsection{Graded groups}\label{subs08291614}
All Lie algebras considered here are over $\R$ and finite-dimensional.
\begin{Def}[Graded group]
	A Lie algebra $\g$ is \emph{graded} if it is equipped with \emph{grading}, i.e., with a vector-space decomposition 
	$\g=\bigoplus_{i>0} V_i$, where $i>0$ means $i\in(0,\infty)$, 
	such that for all $i,j>0$ it holds $[V_i,V_j]\subset V_{i+j}$.
	A \emph{graded Lie group} is a simply connected Lie group $G$ whose Lie algebra is graded.
	The \emph{maximal degree} of a graded group $G$ is the maximum $i$ such that $V_i\neq\{0\}$.
\end{Def}
Graded groups are nilpotent and the exponential map $\exp:\g\to G$ is a global diffeomorphism.
We will denote by $0$ the neutral element of $G$ and identify $\g = T_0G$.
\begin{Def}[Dilations]
In a graded group for which the Lie algebra has the grading
$\g=\bigoplus_{i>0} V_i$,
	the \emph{dilations} 
	relative to the grading are the 
	group homomorphisms 
	$\delta_\lambda:G\to G$, for $\lambda\in(0,\infty)$, 
such that $(\delta_\lambda)_*(v)=\lambda^iv$ for all $v\in V_i$.	
\end{Def}
In the definition above, $\phi_*$ denotes the Lie algebra homomorphism associated to a Lie group homomorphism $\phi$, 
in particular,
$\phi \circ \exp = \exp \circ \phi_*$.
Since a graded group is simply connected,
$\delta_\lambda$ is well defined. 
Notice that, for any $\lambda,\mu>0$, $\delta_\lambda\circ\delta_\mu = \delta_{\lambda\mu}$.

\begin{Def}[Homogeneous distances]\label{def08291617}
	Let $G$ be a graded group with a dilations $\{\delta_\lambda\}_{\lambda>0}$, relative to the grading.
	We say that a distance $d$ on $G$ is \emph{homogeneous} if it is left-invariant, i.e., for every $g,x,y\in G$ we have $d(gx,gy) = d(x,y)$, and one-homogeneous with respect to the dilations, i.e., for all $\lambda>0$ and all $x,y\in G$ we have $d(\delta_\lambda x,\delta_\lambda y) = \lambda d(x,y)$.
If $d$ is one such a distance, then $(G,d)$ is called {\em homogeneous group} (with dilations relative to the grading).
\end{Def}

\begin{Rem}
A graded group admits a homogeneous distance if and only if  for $i\in(0,1)$ we have $V_i=\{0\}$, see \cite{MR1067309}. 
\end{Rem}

Given a homogeneous distance $d$, the function $p\mapsto d_0(p):=d(0,p)$ is a homogeneous norm.
Here with the term \emph{homogeneous norm} we mean a function $N:G\to[0,+\infty)$ such that for all $p,q\in G$ and all $\lambda>0$ it holds
\begin{enumerate}
\item 	$N(p)=0\ \IFF\ p=0$;
\item 	$N(pq) \le N(p)+N(q)$;
\item 	$N(p^{-1}) = N(p)$;
\item 	$N(\delta_\lambda p) = \lambda N(p)$.
\end{enumerate}
In fact, homogeneous distances are in bijection with homogeneous norms on $G$ through the formula $d(p,q) = N(p^{-1}q)$. 

Homogeneous distances induce the original topology of $G$, see \cite{LeDonne_Rigot_rmknobcp}.
Moreover, given two homogeneous distances $d_1,d_2$ on $G$, there is a constant $C>0$ such that for all $p,q\in G$
\begin{equation}\label{eq09031752}
\frac1C d_1(p,q) \le d_2(p,q) \le C d_1(p,q) .
\end{equation}
\begin{Lem}\label{lem09031753}
	Let $G$ be a graded group and $0<k_1\le k_2$ such that $V_i=\{0\}$ for all $i<k_1$ and all $i>k_2$.
	Let $d$ be a homogeneous distance and $\rho$ a left-invariant Riemannian metric on $G$.
	Then there are $C,\epsilon>0$ such that for all $p,q\in G$ with $\rho(p,q)<\epsilon$ it holds
	\begin{equation}\label{eq05311348}
	\frac1C \rho(p,q)^{\frac1{k_1}} \le d(p,q) \le C\rho(p,q)^{\frac1{k_2}} .
	\end{equation}
	In particular, the homogeneous norm $d_0$ is locally $\frac1{k_2}$-Hölder.
\end{Lem}
\begin{proof}
	We identify $G=\g$ via the exponential map. 
	So, if $p\in G$, we denote by $p_i$ the $i$-th component in the decomposition $p=\sum_ip_i$ with $p_i\in V_i$.
	Fix a norm $|\cdot|$ on $\g$.
	For any pair $(p,q)\in G\times G$ define   
	\[\eta(p,q):=\eta(0,p^{-1}q) , \text{ where }
	\eta(0,p) := \max_i (|p_i|)^{\frac1i} 
	.\]
	The function $\eta$ is a so-called quasi-distance, see \cite{LeDonne_Rigot_rmknobcp}.
	In particular, $\eta$ is continuous, left-invariant and one-homogeneous with respect to the dilations $\delta_\lambda$.
	Therefore, if $d$ is a homogeneous distance, then there is $C>0$ such that
	\[
	\frac1C \eta(p,q) \le d(p,q) \le C \eta(p,q) .
	\]
	So, we can prove \eqref{eq05311348} only for $\eta$. 
	
	Let $C,\epsilon>0$ be with $C\epsilon<1$ and such that, if $\rho(0,p)<\epsilon$, then 
	\begin{equation}\label{eq09062352}
	\frac1C \rho(0,p) \le \max_i |p_i| \le C\rho(0,p) .
	\end{equation}
	Therefore, if $\rho(p,q)<\epsilon$, then $|(p^{-1}q)_i|\le C\rho(p,q)<1$ for all $i$ and
	\begin{equation}\label{eq09062353}
	\max_i |(p^{-1}q)_i|^{\frac1{k_1}} 
	\le  \max_i (|(p^{-1}q)_i|)^{\frac1i} 
	= \eta(p,q)
	\le \max_i |(p^{-1}q)_i|^{\frac1{k_2}}  ,
	\end{equation}
	thanks to the monotonicity of the function $x\mapsto a^x$ for $0<a<1$.
	The thesis follows immediately from \eqref{eq09062352} and \eqref{eq09062353} combined.
\end{proof}

Next lemma gives a characterization of sets that are the unit ball of a homogeneous distance.
In this paper, we denote by $\Int(B)$ the interior part of a subset $B$.
The proof of the following fact is straightforward and hence omitted.
\begin{Lem}\label{lem05282259}
	Let $G$ be a graded group with dilations $\delta_\lambda$, $\lambda>0$.
	A set $B\subset G$ is the unit ball with center $0$ of a homogeneous distance on $G$ 
	if and only if $B$ is compact, $0\in\Int(B)$, $B=B^{-1}$ and
	\begin{equation}\label{eq05291157}
	 	\forall p,q\in B,\ \forall t\in[0,1]\quad\delta_t(p)\delta_{1-t}(q)\in B .
	\end{equation}
\end{Lem}


\begin{Def}[Stratified group]
	A \emph{stratified group} is a graded group $G$ such that its Lie algebra $\g$ is generated by the layer $V_1$ of the grading of $\g$. 
\end{Def}
Notice that in a stratified group $G$ the maximal degree $s$ of the grading equals the nilpotency step of $G$ and
it holds $\g=\bigoplus_{i=1}^s V_i$ with $[V_1,V_i]=V_{i+1}$ for all $i\in\{1,\dots,s\}$, with $V_{s+1} = \{0\}$.
We also remark that all stratifications of a group $G$ are isomorphic to each other, i.e., if $\g=\bigoplus_{i=1}^{s'} W_i$ is a second stratification, then there is a Lie group automorphism $\phi:G\to G$ such that $\phi_*(W_i)=V_i$ for all $i$, see \cite{Le-Donne2015A-course-on-Car}.

In a stratified group, the map $\f:G\times V_1\to TG$, 
$
\f(g,v) := \dd L_g(v)
$, 
is a bundle morphism with $\f(g,v)\in T_gG$.
So, if $\|\cdot\|$ is any norm on $V_1$, the triple $(V_1,\|\cdot\|,\f)$ is a sub-Finsler structure on $G$.
The stratified group $G$ endowed with the corresponding sub-Finsler distance $d$ is called \emph{Carnot group}.
Such a $d$ is an example of a homogeneous distance on $G$. 

\begin{Rem}\label{rem04241215}
	As already stated, singular curves play a central role in our analysis, because they disrupt the Lipschitz regularity of the distance function.
	We remark that every Carnot group of nilpotency step $s\ge3$ has singular geodesics.
	More precisely, there is $X\in V_1$ such that the curve $t\mapsto \exp(tX)$ is a singular geodesic.
	In particular, if all non-constant length-minimizing curves are regular, then the step of the group is necessarily at most 2.
\end{Rem}


\section{Regularity of sub-Finsler distances}\label{sec08291633}

We will prove in this section that sub-Finsler distances are Lipschitz whenever all length-minimizing 
 curves are regular, see Theorem~\ref{teo04241155}.
Theorem~\ref{teo04221818} expresses this result for Carnot groups.

It is important 
to remind
what
is known
in the sub-Riemannian case.
A \emph{sub-Riemannian distance} is a sub-Finsler distance whose norm on the bundle $\E$ is induced by a scalar product.
Rifford proved in \cite{MR2307687} that,
if there are no singular length-minimizers,
for all $o\in M$, 
not only $d_o$ is locally Lipschitz, but also 
 the spheres centered at $o$ are Lipschitz hypersurfaces for almost all radii.
The key points of his proof are the tools of 
Clarke's non-smooth calculus (see \cite{MR1488695}) 
and a version of Sard's Lemma for the distance function (see \cite{MR2128549}).
An 
exhaustive exposition of this topic 
can be found in \cite{Agrachev2014Introduction-to}.

In Rifford's version of Sard's Lemma, 
one uses the fact that
the $\LL^2$~norm in
the Hilbert space $\LL^2([0,1];\E)$
is smooth away from the origin.
If $\E$ is equipped with a generic norm, instead, 
the $\LL^p$ norm on $\LL^p([0,1];\E)$ with $1\le p\le \infty$
may be non-smooth,
hence
the proof does not work in the sub-Finsler case.

The non-smoothness of the 
norm can be seen in another dissimilarity between sub-Riemannian and sub-Finsler distances.
%
Sub-Riemannian distances are proven to be locally semi-concave when there are no singular length-minimizing curves.
We remind that a function $f:\R^n\to\R$ is \emph{semi-concave} if for each $p\in \R^n$ there exists a $\Co^2$ function $g:\R^n\to\R$ such that $f\le g$ and $f(p)=g(p)$, see \cite{MR3308395}.
Semi-concavity is a stronger property than being Lipschitz.
However, semi-concavity fails to hold 
in the sub-Finsler case.
For example, the $\ell^1$-distance $d(0,(x,y)) := |x|+|y|$ on $\R^2$ is a sub-Finsler distance that is not semi-concave along the coordinate axis, although all curves are regular.


%
%

We restrict our analysis to the Lipschitz regularity of the distance function, from which we deduce 
 regularity properties of the spheres 
 by means of the homogeneity of Carnot groups.
With this aim in view,
the core of the proof of Theorem~\ref{teo04241155} 
is the bound on the point-wise Lipschitz constant 
(see \eqref{eq1013} at page \pageref{eq1013}),
which already appeared in the sub-Riemannian contest, see \cite{MR2513150}.
Our approach differs from the sub-Riemannian case for the fact that the set of 
optimal curves joining two points
on a sub-Finsler manifold may not be compact 
in the $W^{1,\infty}$ topology.
As an example, consider the set of all length-minimizers from $(0,0)$ to $(0,1)$ for the $\ell^\infty$-distance $d(0,(x,y)) := \max\{|x|,|y|\}$ on $\R^2$. 
However, we are still able to obtain a bound on the pointwise Lipschitz constant, i.e., to prove \eqref{eq1013}, by use of the weak* topology on controls.

\begin{Teo}\label{teo04241155}
	Let $(\E,\|\cdot\|,\f)$ be a sub-Finsler structure on $M$ with sub-Finsler distance $d$.
	Fix $o$ and $p$ in $M$.
	If all the length-minimizing curves from $o$ to $p$ are regular, then for every Riemannian metric $\rho$ on $M$ there are a neighborhood $U$ of $p$ and $L>0$ such that 
	\begin{equation}\label{eq09072111}
	\forall q_1,q_2\in U
	\qquad
	d_o(q_1)-d_o(q_2) \le L\rho(q_1,q_2) .
	\end{equation}
\end{Teo}
The proof is presented in Section~\ref{subs07122224}.
\begin{Rem}
	Theorem~\ref{teo04241155} can be made more quantitative.
	Define
	\[
	\tau_0 := \inf\left\{\tau(\dd\End_o(u)):\End_o(u)=p\text{ and }\|u\|_{\LI} = d(o,p) \right\} ,
	\]
	where $\tau$ the minimal stretching, which we will recall in Definition~\ref{def05111117}.
	Then, for every $L>\frac{1}{\tau_0}$, there exists a neighborhood $U$ of $p$ such that \eqref{eq09072111} holds. 
	The hypothesis of regularity of all length-minimizing curves from $o$ to $p$ is equivalent to $\tau_0>0$.
\end{Rem}

In the case of Carnot groups (of step 2, see Remark~\ref{rem04241215}), we can obtain a more global result
\begin{Prop}\label{prop06031443}
	Let $(G,d)$ be a Carnot group   without non-constant singular geodesics. 
	Then for every left-invariant Riemannian metric $\rho$ and every neighborhood $U$ of $0$ the function $d_0:x\mapsto d(0,x)$ is Lipschitz on $G\setminus U$.
	Moreover, the function $d_0^2:x\mapsto d(0,x)^2$ is Lipschitz in a neighborhood of $0$.
\end{Prop}
\begin{proof}
	Thanks to Theorem~\ref{teo04241155}, one easily shows that there are $L>0$ and an open neighborhood $\Omega$ of the unit sphere $\{p:d(0,p)=1\}$ such that $d_0$ is $L$-Lipschitz on $\Omega$.
	
	Next, we claim that $d_0$ is locally $L$-Lipschitz on $G\setminus B_d(0,1)$.
	Indeed, let $r>0$ be such that $B_\rho(x,r)\subset\Omega$ for all $x\in S_d(0,1)$.
	If $q_1,q_2\in G\setminus B_d(0,1)$ are such that $\rho(q_1,q_2)<r$, 
	then there is $o\in G$ such that $d(0,q_1)=d(0,o)+d(o,q_1)$ and $d(o,q_1)=1$, therefore 
	\[
	d_0(q_2)-d_0(q_1)
	\le d(o,q_2) - d(o,q_1) 
	\le L \rho(o^{-1}q_2,o^{-1}q_1) 
	= L \rho(q_2,q_1) .
	\]
		
	In the second step of the proof, we prove that $d_0$ is $L$-Lipschitz on $G\setminus B_d(0,1)$.
	Let $p,q\in G\setminus B_d(0,1)$ and
	 $\gamma:[0,1]\to G$ a $\rho$-length minimizing curve from $p$ to $q$.
	If $\Im\gamma\subset G\setminus B_d(0,1)$, then there are $0=t_0\le t_1\le\dots\le  t_{k+1}=1$ such that $d_0(\gamma(t_i)) - d_0(\gamma(t_{i+1})) \le L \rho(\gamma(t_i),\gamma(t_{i+1}))$ for all $i$.
	Hence
	\begin{align*}
	d_0(p)-d_0(q)
	&= \sum_{i=0}^k d_0(\gamma(t_i)) - d_0(\gamma(t_{i+1})) \\
	&\le L \sum_{i=0}^k \rho(\gamma(t_i),\gamma(t_{i+1}))
	= L \rho(p,q) .
	\end{align*}
	If instead $\Im\gamma\cap B_d(0,1)\neq\emptyset$, then there are $0<s<t<1$ such that $d_0(\gamma(s))=d_0(\gamma(t)) = 1$ and $\gamma([0,s])\subset G\setminus B_d(0,1)$ and $\gamma([t,1])\subset G\setminus B_d(0,1)$.
	Then 
	\begin{multline*}
	d_0(p)-d_0(q) 
	= d_0(p) - d_0(\gamma(s)) 
		+ d_0(\gamma(t)) - d_0(q) \\
	\le L \left(\rho(p,\gamma(s)) 
	+ \rho(\gamma(t),q) \right)
	\le L \rho(p,q) .
	\end{multline*}

%
%
	
	Finally, let $p,q\in G\setminus B_d(0,r)$ for $0<r<1$.
	Then $\delta_{r^{-1}}p, \delta_{r^{-1}}q\in G\setminus B_d(0,1)$ and
	we have
	\begin{equation}\label{eq06041522}
	d_0(p)-d_0(q) 
	= r (d_0(\delta_{r^{-1}}p)-d_0(\delta_{r^{-1}}q) )
	\le Lr \rho(\delta_{r^{-1}}p,\delta_{r^{-1}}q)
	\le \frac{CL}{r} \rho(p,q) ,
	\end{equation}
%
	where we used in the last step the fact that there exists $C>0$ such that
	\[
	\forall p,q\in G,\ \forall r\in(0,1) \qquad\rho(\delta_{r^{-1}}p,\delta_{r^{-1}}q) \le C r^{-2} \rho(p,q).
	\]

	
	Now, we need to prove that $d_0^2$ is Lipschitz on $B_d(0,1)$.
	We first claim that $d_0^2$ is locally $4L$-Lipschitz on $B_d(0,1)\setminus\{0\}$.
	Indeed, if $p,q\in B_d(0,1)\setminus\{0\}$ are such that 
	\[
	\frac12 \le \frac{d_0(p)}{d_0(q)} \le 2 ,
	\]
	then
	\[
	0< d_0(p)+d_0(q) \le 4\min\{d_0(p),d_0(q)\} .
	\]
	Therefore, using \eqref{eq06041522},
	\begin{multline*}
	d_0(p)^2 - d_0(q)^2
	= (d_0(p)+d_0(q)) (d_0(p)-d_0(q)) \\
	\le 	(d_0(p)+d_0(q)) \frac{CL}{\min\{d_0(p),d_0(q)\} } \rho(p,q) \\
	\le (d_0(p)+d_0(q)) \frac{4CL}{d_0(p)+d_0(q)} \rho(p,q)
	= 4CL \rho(p,q).
	\end{multline*}
	Finally, using again the fact that $\rho$ is a geodesic distance, we get that $d_0^2$ is $4CL$ Lipschitz on $B_d(0,1)\setminus\{0\}$ and therefore on $B_d(0,1)$.
\end{proof}

\subsection{About the minimal stretching}
\begin{Def}[Minimal Stretching]\label{def05111117}
	Let $(X,\|\cdot\|)$ and $(Y,\|\cdot\|)$ be normed vector spaces.
	We define for a continuous linear map $L:X\to Y$ the \emph{minimal stretching}
	\[
	\str(L) := \inf\{\|y\|:y\in Y\setminus L(B_X(0,1))\}
	\]
	where $B_X(p,r)=\{q\in X: \|q-p\|<r\}$.
\end{Def}
It is easy to prove that $\tau:L(X;Y)\to[0,+\infty)$ is continuous, where $L(X;Y)$ is the space of continuous linear mappings $X\to Y$ endowed with the operator norm.

The next proposition applies this notion to smooth functions.
We omit the proof because it follows the standard arguments for the Implicit Function Theorem.
\begin{Prop}\label{prop06021609}
	Let $(X,\|\cdot\|)$ and $(Y,\|\cdot\|)$ be two Banach spaces and $F:\Omega\to Y$ a $\Co^2$ map, where $\Omega\subset X$ is open.
	Fix $\hat x\in \Omega$ and let $\tau_0:=\str(\dd F(\hat x))>0$.
	Then for every $C>1$ there is $\hat \epsilon>0$, such that for all $0<\epsilon<\hat \epsilon$ it holds
	\[
	B(F(\hat x),\epsilon) \subset F\left(B(\hat x,\frac C{\tau_0}\epsilon)\right) .
	\]
\end{Prop}

\subsection{The End-point map is weakly* continuous}

As before, let $(\E,\f,\|\cdot\|)$ be a sub-Finsler structure on a manifold $M$.
We want to prove the following proposition.
\begin{Prop}\label{prop05111227}
	Fix $o\in M$ and let $o_k\in M$ be a sequence converging to $o$.
	Let $u_k\in\LUI$ be a sequence of controls weakly* converging to $u\in\LUI$.	
	Let $\gamma_k$ (resp.~$\gamma$) be the curve with control $u_k$ (resp.~$u$) and $\gamma_k(0)=o_k$ (resp.~$\gamma(0)=o$).
	Then $\gamma_k$ uniformly converge to $\gamma$.
\end{Prop}
In particular, it follows that the End-point map $\End_o:\LUI\to M$ is weakly* continuous.
\begin{proof}
	Since the sequence $u_k$ is bounded in $\LUI$ by the Banach–Steinhaus Theorem and the sequence $o_k$ is bounded in $(M,d)$, then there is a compact set $K\subset M$ such that $\gamma_k\subset K$ for all $k$.
	Let $R>0$  be such that $\|u_k\|_{\LI}\le R$ for all $k\in\N$.
	
	Thanks to the Whitney Embedding Theorem, we can assume that $M$ is a submanifold of $\R^N$ for some $N\in\N$.
	Fix a basis $e_1,\dots,e_r$ of $\E$ and define the vector fields $X_i:M\to\R^N$ as
	\[
	X_i(p) := \f(p,e_i) .
	\]
	Since they are smooth, they are $L$-Lipschitz on $K$ for some $L>0$.
	We extend the vector fields $X_i:M\to\R^N$ to smooth functions $X_i:\R^N\to\R^N$.
	
	Define $\eta_k:[0,1]\to \R^N$ as 
	\[
	\eta_k(t) := o_k + \int_0^t u_k^i(s) X_i(\gamma(s))\dd s .
	\]
	Since $t\mapsto  X_i(\gamma(t))\in\R^N$ are continuous, then $u_k^i X_i(\gamma)\wsto u^i  X_i(\gamma)$, for all $i\in\{1,\dots,r\}$.
	In particular, $\eta_k(t)\to \gamma(t)$ for each $t\in[0,1]$.
	Moreover, since the $\eta_k$'s have uniformly bounded derivative, they are a pre-compact family of curves with respect to the topology of uniform convergence. 
	This fact and the pointwise convergence imply that $\eta_k\to \gamma$ uniformly on $[0,1]$.
	
	Set $\epsilon_k:=\sup_{t\in[0,1]}|\eta_k(t)-\gamma(t)| + 2|o_k-o|$, so that $\epsilon_k\to 0$, where $|\cdot|$ is the usual norm in $\R^N$.
	
	By Ascoli-Arzelà Theorem, the family of curves $\{\gamma_k\}_k$ is also precompact with respect to the uniform convergence.
	Hence, if we prove that the only accumulation curve of $\{\gamma_k\}_k$ is $\gamma$, then we obtain that $\gamma_k$ uniformly converges to $\gamma$.
	So, we can assume $\gamma_k\to\xi$ uniformly for some $\xi:[0,1]\to M$.
	Then we have (sums on $i$ are hidden)
	\begin{align*}
	 	|\gamma_k(t)-\gamma(t)| 
		&\le |o_k-o| + \left| \int_0^t u_k^i(s) X_i(\gamma_k(s)) - u^i(s)  X_i(\gamma(s)) \dd s \right| \hfill  \\
		&\le |o_k-o| + \int_0^t |u_k^i(s) X_i(\gamma_k(s)) - u_k^i(s) X_i(\gamma(s))| \dd s +\hfill \\
		&\qquad\qquad	 + \left| \int_0^t u_k^i(s) X_i(\gamma(s)) - u^i(s)  X_i(\gamma(s))  \dd s \right|  \\
		&\le 2|o_k-o| + rRL \int_0^t|\gamma_k(s)-\gamma(s)| \dd s + |\eta_k(t)-\gamma(t)| \\ 
		&\le rRL \int_0^t|\gamma_k(s)-\gamma(s)| \dd s + \epsilon_k .
	\end{align*}
	Passing to the limit $k\to\infty$, we get for all $t\in[0,1]$
	\begin{equation}\label{eq1250}
	|\xi(t)-\gamma(t)| \le rRL \int_0^t |\xi(s)-\gamma(s)| \dd s .
	\end{equation}
	Starting with the fact that $\|\xi-\gamma\|_{\LI}\le C$ for some $C>0$ and iterating the previous inequality, we claim that
	\[
	|\xi(t)-\gamma(t)| \le C \frac{(rRLt)^j}{j!} 
	\qquad\forall j\in\N,\ \forall t\in[0,1] .
	\]
	Indeed, by induction, from \eqref{eq1250} we get
	\[
	 	|\xi(t)-\gamma(t)|
		\le rRL \int_0^t C \frac{(rRL)^j}{j!} t^j \dd s 
		= C \frac{(rRL)^{j+1}}{j!} \frac{t^{j+1}}{j+1} .
	\]
	Finally, since $\lim_{j\to\infty}\frac{(rRLt)^j}{j!} = 0$, 
	we have $|\xi(t)-\gamma(t)| = 0$ for all $t$.
\end{proof}

\subsection{The differential of the End-point map is an End-point map}
The End-point map behaves like the exponential map: its differential is again an End-point map.
In order to make this statement precise, we consider the case $M=\R^n$.
Notice that we don't need any bracket-generating condition.
In Corollary~\ref{cor06021214} we will use the results on $\R^n$ to prove a statement for all manifolds.

Let $\f:\R^n\times\E\to \R^n$ be a smooth map.
Given a basis $e_1,\dots,e_r$ of $\E$, we define the vector fields $X_i:\R^n\to\R^n$ as
\[
X_i(p) := \f(p,e_i) .
\]
The differential of the End-point map with basis point $0$ is the map 
\[
\begin{array}{rccc}
 	\dd\End_0:&\LUI\times\LI([0,1];\E) &\to &\R^n  \\
	&(u,v) &\mapsto &\dd\End_0(u)[v] .
\end{array}
\]
Define $Y_i,Z_i:\R^n\times\R^n\to\R^n\times\R^n$ for $i=1,\dots,r$ as
\[
\begin{cases}
 	Y_i(p,q) := \left( X_i(p),\dd X_i(p)[q] \right) \\
	Z_i(p,q) := \left( 0,X_i(p) \right)
\end{cases}
\]
where $\dd X_i(p):\R^n\to\R^n$ is the differential of $X_i$ at $p$.
These vector fields induce a new End-point map
\[
\End_{00}:\LI([0,1];\E\times\E)\to\R^n\times\R^n 
\]
with starting point $(0,0)\in \R^n\times\R^n$.

\begin{Prop}\label{prop06021149}
 	For all $u,v\in\LI([0,1];\E)$ it holds
	\[
	\left(\End_0(u) , \dd\End_0(u)[v] \right) = \End_{00}(u,v) .
	\]
\end{Prop}


The proof is immediate, and hence omitted, once one has an explicit representation of the differential $\dd\End_0(u)[v]$, see \cite{Mon02A-tour-of-subriemannian-geometries-their-geodesics-and-applications}.
This result, together with Proposition~\ref{prop05111227}, gives us the weakly* continuity of the differential of the End-point map.
The next corollary is an application.
\begin{Cor}\label{cor06021214}
	Let $(\E,\f,\|\cdot\|)$ be a sub-Finsler structure on a manifold $M$ and $o\in M$.
	Let $\rho$ be a Riemannian metric on $M$.
	Then the map $\LUI\to[0,+\infty)$, $u\mapsto \tau(\dd\End_o(u))$ is weakly* lower semi-continuous, where $\tau$ is the minimal stretching computed with respect to the norm given by $\rho$.
\end{Cor}
\begin{proof}
%
	Let $\{u_k\}_k\subset\LUI$ be a sequence such that $u_k\wsto u\in\LUI$.
	Let $\gamma_u$ be the curve with control $u$ and starting point~$o$.
	We can pull back the sub-Finsler structure from a neighborhood of $\gamma_u$ to an open subset of $\R^n$ via a covering map,
	so that we reduce the statement to the case $M=\R^n$.
%
%

	We need to prove 
	\begin{equation}\label{eq06021142}
		\liminf_{k\to\infty} \tau(\dd\End_o(u_k)) \ge \tau(\dd\End_o(u)) .
	\end{equation}
	Set $\hat\tau:=\tau(\dd\End_0(u))$.
	If $\hat\tau=0$, then \eqref{eq06021142} is fulfilled, so we assume $\hat\tau>0$.
	Let $\hat\tau> \epsilon>0$.
	Then there exists a finite-dimensional vector space $W\subset\LI([0,1];\E)$ such that 
	\[
	B_{\End_0(u)}(0,\hat\tau-\frac\epsilon2) \ssubset \dd\End_0(u)[B_{\LI}(0,1)\cap W] ,
	\]
	where, for $p\in\R^n$, $B_{p}$ denotes a ball in $\R^n=T_{p}\R^n$ with respect to the norm given by $\rho$ at $p$, and $B_{\LI}$ denotes a ball in $\LI([0,1];\E)$ with respect to the $\LI$-norm induced by $\|\cdot\|$.
	Since $\dim W<\infty$ and by Propositions~\ref{prop05111227} and~\ref{prop06021149}, the maps $\dd\End_0(u_k)|_W$ strongly converge to $\dd\End_0(u)|_W$.
	Moreover, the norm on $\R^n=T_{p}\R^n$ given by $\rho$ continuously depends on $p\in\R^n$.
	Therefore, for $k$ large enough we have
	\[
	B_{\End_0(u_k)}(0,\hat\tau-\epsilon) \subset \dd\End_0(u_k)[B_{\LI}(0,1)\cap W] .
	\]
	Hence
	\[
	\liminf_{k\to\infty}\tau(\dd\End_0(u_k)) \ge \hat\tau-\epsilon .
	\]
	Since $\epsilon$ is arbitrary, \eqref{eq06021142} follows.
\end{proof}

\subsection{The sub-Finsler distance is Lipschitz in absence of singular geodesics}\label{subs07122224}
The proof of Theorem~\ref{teo04241155} is divided into the next two lemmas from which Theorem~\ref{teo04241155} follows.

\begin{Lem}\label{lem06021600}
	Let $o,p\in M$ such that all $d$-minimizing curves from $o$ to $p$ are regular.
	Then there exist a compact neighborhood $K\subset M$ of $p$ and a weakly* compact set $\scr K\subset\LUI$
	such that:
	\begin{enumerate}
	\item 	$\End_o:\scr K\to K$ is onto;
	\item 	$\dcc(o,\End_o(u))=\|u\|_{\LI}$ for all $u\in\scr K$;
	\item 	every $u\in \scr K$ is a regular point for $\End_o$.
	\end{enumerate}
\end{Lem}
\begin{proof}
	For any compact neighborhood $K$ of $p$, define the compact set
	\[
	\scr K(K) := \left\{u\in\LUI:\End_o(u)\in K\text{ and }d(o,\End_o(u)) = \|u\|_{\LI} \right\} .
	\]
	Since the metric $d$ is geodesic, the End-point map $\End_o:\scr K(K)\to K$ is surjective, for all $K$.
	Moreover, the second requirement holds by definition.
	Finally, suppose that there exist a sequence $p_k\to p$ and a sequence $u_k\in\LUI$ such that $\End_o(u_k)=p_k$, $d(0,p_k)=\|u_k\|_{\LI}$ and $u_k$ is a singular point for $\End_o$, for all $k$.
	Since the sequence $u_k$ is bounded, thanks to the Banach–Alaoglu Theorem there is $u\in\LUI$ such that, up to a subsequence, $u_k\wsto u$.
	By the continuity of $\End_o$, we have $\End_o(u)=p$.
	By Corollary~\ref{cor06021214}, we have $\tau(\dd\End_o(u)) \le \liminf_k\tau(\dd\End_o(u_k)) = 0$.
	Finally, by the lower-semicontinuity of the norm with respect to the weak* topology, we have
	\[
	\|u\|_{\LI} 
	\le \liminf_{k\to\infty}\|u_k\|_{\LI}
	= \liminf_{k\to\infty} d(o,p_k)
	= d(o,p)
	\le \|u\|_{\LI} .
	\]
	So, $u$ is the control of a singular length-minimizing curve from $o$ to $p$, against the assumption.
	Therefore, there exists a neighborhood $K$ of $p$ such that $\scr K(K)$ contains only regular points for $\End_o$.
\end{proof}

\begin{Lem}\label{lem05251454}
	Let $o\in M$, and $K\subset M$ compact.
	Suppose there is a weakly* compact set $\scr K\subset\LUI$
	that satisfies all three properties listed in Lemma~\ref{lem06021600}.
	Then for every Riemannian metric $\rho$ on $M$ there exists $L>0$ such that the function $d_o:p\mapsto d(o,p)$ is locally $L$-Lipschitz on the interior of $K$.
\end{Lem}
\begin{proof}
	Let $\rho$ be a Riemannian metric on $M$.
	We first prove that 
	\begin{equation}\label{eq1013}
	\begin{array}{c}
	 	\exists L>0,\ \forall q\in K,\ \exists\hat\epsilon_q>0\ \forall q'\in K\\
		\big[\rho(q,q')<\hat\epsilon_q\THEN 
		d_o(q')-d_o(q) \le L\rho(q,q') \big] .
	\end{array}
	\end{equation}
	Since $\scr K$ is a weakly* compact set of regular points for $\End_o$, then by Corollary~\ref{cor06021214} the function $u\mapsto \tau(\dd\End_o(u))$
	 admits a minimum on $\scr K$ that is strictly positive.
	By Proposition~\ref{prop06021609}, there is $L>0$ such that for every $u\in\scr K$ there is $\hat\epsilon_u>0$ such that
	\begin{equation}\label{eq2255}
	 	B_\rho\left( \End_o(u),\epsilon \right) \subset \End_o\left( B_{\LI}(u,L\epsilon) \right)
	\quad\forall \epsilon<\hat\epsilon_u.
	\end{equation}
	Let $q,q'\in K$ be such that $q=\End_o(u)$ with $u\in\scr K$ and $\epsilon := \rho(q,q') <\hat\epsilon_u$. 
	Then, by the inclusion \eqref{eq2255}, there is $u'\in B_{\LI}(u,L\epsilon)$ with $\End_o(u')=q'$.
	So
	\[
	d_o(q')-d_o(q) 
	\le \|u'\|_{\LI} - \|u\|_{\LI} 
	\le \|u'-u\|_{\LI} 
	\le L\epsilon 
	= L\rho(q,q') .
	\]
	This proves the claim \eqref{eq1013}.
	
	Finally, if $p$ is an interior point of $K$, then there is a $\rho$-convex neighborhood $U$ of $p$ contained in $K$ (see \cite{opac-b1127541}).
	So, if $q,q'\in U$, then there is a $\rho$-geodesic $\xi:[0,1]\to U$ joining $q$ to $q'$.
	Since the image of $\xi$ is compact, there are $0= t_1<s_1<t_2<s_2<\dots<s_{k-1}<t_{k}=1$ such that 
	$\rho(\xi(t_i),\xi(s_i)) < \hat\epsilon_{\xi(t_i)}$ and
	$\rho(\xi(s_i),\xi(t_{i+1})) < \hat\epsilon_{\xi(t_{i+1})}$.
	Therefore
	\begin{align*}
	d_o(q')-d_o(q) 
	&\le \sum_{i=1}^{k-1}d_o(\xi(t_{i+1}))-d_o(\xi(s_i)) + d_o(\xi(s_i)) - d_o(\xi(t_{i})) \\
	&\le L\sum_{i=1}^{k-1} \rho(\xi(t_{i+1}),\xi(s_i)) + \rho( \xi(s_i) ,\xi(t_i))
	= L \rho(q,q').
	\end{align*}
\end{proof}

\section{Regularity of spheres in graded groups}\label{sec05111258}
This section is devoted to the proof of Theorems~\ref{teo04221825} and~\ref{teo04221835} and of the inequalities \eqref{eq04282051}.

\subsection{The sphere as a graph}
Let $(G,d)$ be a homogeneous group and $\rho$ a Riemannian distance on $G$.
Theorem~\ref{teo04221825} and the estimate \eqref{eq04282051} are both based on the following remark.
\begin{Rem}\label{rem09051148}
	Let $\g=\oplus_{i>0}V_i$ be a grading for the Lie algebra of $G$.
	Let $|\cdot|$ be the norm of a scalar product on $\g$ that makes the layers orthogonal to each other
	and let $S=\exp(\{v:|v|=1\})\subset G$.
	The hypersurface $S$ is smooth and transversal to the dilations, i.e., for all $p\in S$ we have
	$
	\frac{\dd}{\dd t}|_{t=1} \delta_tp \notin T_pS .
	$
	Define
	\[
	\begin{array}{rccc}
	 	\phi: &S\times(0,+\infty) &\to &G\setminus\{0\}	\\
		& (p,t) &\mapsto &\delta_{\frac1t}p .
	\end{array}
	\]
	Since $S$ is transversal to the dilations, $\phi$ is a diffeomorphism.
	Moreover, if $\Gamma :=\{(p,d_0(p)):p\in S\}\subset S\times(0,+\infty)$ is the graph of the function $d_0$ restricted to $S$, then 
	\[
	\bb S_d = \phi(\Gamma) .
	\]
\end{Rem}

Thanks to 
 the last remark, 
the estimate \eqref{eq04282051} follows from the next lemma. 

\begin{Lem}\label{lem04251012}
	Let $\Omega\subset\R^n$ be an open set and let $f:\Omega\to\R$ be an $\alpha$-Hölder function, i.e., for all $x,y\in\Omega$ we have
	\[
	|f(x)-f(y)| \le C |x-y|^\alpha , 
	\]
	for some $C>0$, where $\alpha\in(0,1]$.	
	Define the graph of $f$ as
	\[
	\Gamma_f := \{(x,f(x)):x\in\Omega\}\subset\R^{n+1} .
	\]	
	Then
	\[
	n \le \dim_H\Gamma_f \le n+1-\alpha ,
	\]
	where $\dim_H$ is the Hausdorff dimension.
	Moreover, this estimate is sharp, i.e., there exists $f$ such that $\dim_H\Gamma_f=n+1-\alpha$.
\end{Lem}
The proof is straightforward by use of a simple covering argument or by an estimate of the Minkowski content of the  graph.
The sharpness of this result has been shown in \cite{Besicovitch1937Sets-of-Fractio} for the case $n=1$.
The general case, as stated here, is a simple consequence.
Indeed, if $g:(0,1)\to\R$ is a $\alpha$-Hölder function such that $\dim_H(\Gamma_g)=2-\alpha$, then the graph of the function $f(x_1,\dots,x_n):=g(x_1)$ is $\Gamma_f=\Gamma_g\times(0,1)^{n-1}$.
Therefore, $\dim_H(\Gamma_f) = n+1-\alpha$.


 In the next easy-to-prove lemma we point out that a homogenous distance is locally Lipschitz if and only if the spheres are Lipschitz graphs in the directions of the dilations.
\begin{Lem}\label{lem09051217}
	Let $d$ be a homogeneous distance on $G$. 
	Let $S$ and $\bb S_d$ be as in Remark~\ref{rem09051148} and $p\in S$.
	Then the following conditions are equivalent:
	\begin{enumerate}[label=(\roman*)]
	\item 	Setting $\hat p:=\delta_{d_0(p)^{-1}}(p)\in \bb S_d$, the sphere $\bb S_d$ is
	a Lipschitz graph in 
	the direction $\bar\delta(\hat p)$
	in some neighborhood of $p$;
	\item 	$d_0|_S:S\to(0,+\infty)$ is Lipschitz in some neighborhood of $p$ in~$S$;
	\item 	$d_0$ is Lipschitz in some neighborhood of $\delta_\lambda p$ for one, hence all, $\lambda>0$.
	\end{enumerate}
\end{Lem}


Thanks to Lemma~\ref{lem09051217},
Theorem~\ref{teo04221825} is a consequence of Theorem~\ref{teo04221818}.

\subsection{An intrinsic approach}
In this section we will prove Theorem~\ref{teo04221835}.
We define a \emph{cone} in $\R^n$ as
\[
\Cone(\alpha,h,v) := \{x\in\R^n:|x|\le h\text{ and }\angle(x,v)\le \alpha\}\subset\R^n ,
\]
where $\alpha\in[0,\pi]$, $h\in(0,+\infty]$, $v\in\R^n$ is the axis of the cone, and $\angle(x,v)$ is the angle between $x$ and $v$.
%
%
The following lemma is a simple calculus exercise and it will be used later in the proof of Theorem~\ref{teo04221835}.

\begin{Lem}\label{lem05291640}
 	Let $m,k,n\in\N$,
	$p\in\R^m$ and $y_0\in\R^k$.
	Let $\phi:\R^m\times\R^k\to\R^n$ be a smooth map such that $\dd(\phi_p)(y_0):\R^k\to\R^n$ is surjective, where $\phi_x(y):=\phi(x,y)$.	
	Let $C'\subset\R^k$ be a cone with axis $v'\in\R^k$.
	Then there exist a cone $C\subset\R^n$ with axis $\dd(\phi_p)(y_0)v'$ and an open neighborhood $U\subset\R^m$ of $p$ such that for all $q\in U$
	\[
	\phi_q(y_0)+C \subset \phi_q(y_0+C') .
	\]
\end{Lem} 
\begin{proof}[Proof of Theorem~\ref{teo04221835}]
	In this proof, we consider the dilations $\delta_\lambda$ as defined for $\lambda\le0$ too, with the same definition as for $\lambda>0$.
	Notice that in this way the map $G\times\R\to G$, $(p,\lambda)\mapsto \delta_\lambda p$, is a smooth map.
		
 	Let $v_1,\dots,v_r$ be a basis for $V_1$ and set $p_i :=\exp(v_i)\in G$.
	Up to a rescaling, we can assume $d_0(p_i)<1$ for all $i$.	
	For $p\in G$ define $\phi_p:\R^{2r+1}\to G$ as
	\[
	\phi_p(u_1,\dots,u_r,s,v_1,\dots,v_r)
	= \delta_{u_1}p_1\cdots\delta_{u_r}p_r\cdot\delta_sp\cdot\delta_{v_1}p_1\cdots\delta_{v_r}p_r .
	\]
	Let $\hat x\in\R^{2r+1}$ be the point with $u_i=0$, $s=1$ and $v_i=0$, so that $\phi_p(\hat x) = p$.
	The differential of $\phi_p$ at $\hat x$ is given by the partial derivatives
	\begin{align*}
	\frac{\de\phi_p}{\de u_i}(\hat x) 
	&= \frac{\dd}{\dd t}|_{t=0}\left(\delta_tp_i\cdot p\right) 
	= \dd R_p\left(\frac{\dd}{\dd t}|_{t=0}(\delta_tp_i) \right)
	= \dd R_p(v_i) , \\
	\frac{\de\phi_p}{\de s}(\hat x) 
	&= \frac{\dd}{\dd t}|_{t=1}\left(\delta_tp\right) 
	= \bar\delta(p) , \\
	\frac{\de\phi_p}{\de v_i}(\hat x) 
	&= \frac{\dd}{\dd t}|_{t=0}\left(p\cdot\delta_tp_i\right) 
	= \dd L_p\left(\frac{\dd}{\dd t}|_{t=0}(\delta_tp_i) \right)
	= \dd L_p(v_i) .
	\end{align*}
	Therefore, if $p\in\bb S_d$ is such that the condition \eqref{eq04221835} is true, then the differential $\dd\phi_p$ has full rank at $\hat x$, hence in a neighborhood of $\hat x$.
	
	Define
	\[
	\Delta := \{(u_1,\dots,u_r,s,v_1,\dots,v_r)\in\R^{2r+1}:s+\sum_{i=1}^r(|u_i|+|v_i|) \le 1\} .
	\]
	We identify $G$ with $\R^n$ through an arbitrary diffeomorphism.
	So, by Lemma~\ref{lem05291640}, there is a cone $C$ with axis $\bar\delta(p)$ and a neighborhood $U$ of $p$ such that  for all $q\in U$
	\[
	q+C \subset \phi_q(\Delta) .
	\]
	Up to restricting $U$, for all $q\in U$ there are cones $C_q$ with axis $\bar\delta(q)$, fixed amplitude and fixed height 
	such that $q+C_q\subset q+C$.
	Notice that for all $q\in\bb S_d$ we have $\phi_q(\Delta)\subset \bar B_d(0,1)$ and $\phi_q(\Delta) \cap \bb S_d = \{q\}$. 
	In particular, for all $q\in\bb S_d\cap U$, we have $q+C_q \subset\bar B_d(0,1)$ and $(q+C_q)\cap\bb S_d = \{q\}$, i.e., $\bb S_d\cap U$ is a Lipschitz 
	graph
	in the direction of the dilations.
	Thanks to Lemma~\ref{lem09051217}, we get that $d_0$ is Lipschitz in a neighborhood of $p$.
\end{proof}

Finally, some considerations on condition \eqref{eq04221835} are due.
\begin{Rem}\label{rem07101129}
	If \eqref{eq04221835} holds at $p\in G$ and $u\in\LI([0,1];V_1)$ is a control such that $\End_0(u)=p$, then the differential $\dd\End_0(u)$ is surjective, i.e., $p$ is a regular value of $\End$.
	Indeed, by \cite{LeDonne2015} (see (2.6) and (2.11) there), we have
	\[
	\dd L_p(V_1) + \dd R_p(V_1) + \Span\{\bar \delta(p)\} \subset \Im(\dd\End_0(u)) ,
	\]
	because $q\mapsto\bar\delta(q)$ is a contact vector field of $G$.
\end{Rem}
\begin{Prop}
	Let $X\in V_1$.
 	If \eqref{eq04221835} holds for $p=\exp(X)$, then
	\[
	\g = V_1 + [X,V_1] .
	\]
\end{Prop}
\begin{proof}
	Let $X_1,\dots,X_r$ be a basis for $V_1$ and $Y_1,\dots,Y_\ell$ a basis for $[X,V_1]$.
	Let $\alpha_j^i\in\R$ be such that $[X,X_i]=\sum_{j=1}^\ell \alpha^i_jY_j$.	
	First, notice that		
	\begin{align*}
	T_0G 
	&= \dd L_{\exp(-X)} \left( \dd L_{\exp(X)}(V_1) + \dd R_{\exp(X)}(V_1) \right) \\
	&= V_1 + \dd L_{\exp(-X)} \circ\dd R_{\exp(X)} (V_1) \\
	&= V_1 + \Ad_{\exp(X)} (V_1).
	\end{align*}
	Then, using the formula $\Ad_{\exp(X)}(Y) = e^{\ad_X}(Y) = \sum_{k=0}^\infty\frac1{k!} \ad_X^k(Y)$, we have
	\begin{align*}
	\Ad_{\exp(X)}(X_i) 
	&= X_i + \left( \sum_{k=1}^\infty\frac1{k!} \ad_X^{k-1}([X,X_i]) \right) \\
	&= X_i + \left( \sum_{k=1}^\infty \ad_X^{k-1}(\sum_{j=1}^\ell \alpha_j^iY_j) \right) \\
	&= X_i + \sum_{j=1}^\ell \alpha_j^i \left( \sum_{k=1}^\infty \ad_X^{k-1}(Y_j) \right) .
	\end{align*}
	It follows that $\dim\left(V_1+\Ad_{\exp(X)} V_1\right) \le r + \ell$ and therefore $\dim\g\le r+\ell$, i.e., $\g=V_1+[X,V_1]$.
\end{proof}
\begin{Prop}
	Let $Z\in V_k$, where $k>0$ is such that $V_i=\{0\}$ for all $i>k$.
	If \eqref{eq04221835} holds for $p=\exp(Z)$, then
	\[
	\g = V_1 + \Span\{Z\} .
	\]
\end{Prop}
\begin{proof}
	Since $[Z,\g]=\{0\}$, we have $R_p=L_p$.
	Moreover, $\bar\delta(p)=\dd L_p(kZ)$.
	So, condition \eqref{eq04221835} becomes $\dd L_p (V_1) + \dd L_p (\Span\{Z\}) = T_pG$.
\end{proof}
	In particular, if \eqref{eq04221835} holds for all $p\in G\setminus\{0\}$, then $\g=V_1\oplus V_2$ with $\dim V_2\le1$ and $[X,V_1]=V_2$ for all non-zero $X\in V_1$.

\section{Examples}\label{sec07101128}

\subsection{Three gradings on $\R^2$}
We will present three examples of dilations on $\R^2$.
In particular we want to illustrate two applications of Theorem~\ref{teo04221835} and show the sharpness of the dimension estimate \eqref{eq04282051}.
In Remark~\ref{rem09051840} we give an easy example of a homogeneous distance whose unit ball is a Lipschitz domain, but the distance is not locally Lipschitz away from the diagonal.


The first and the easiest is 
\[
\delta_\lambda(x,y) := (\lambda x,\lambda y),
\]
which gives rise to the known structure of vector space.
Here, homogeneous distances are given by norms and balls are convex, hence Lipschitz domains.
It's trivial to see that condition \eqref{eq04221835} holds for all $p\in\R^2$.

The second example is given by the dilations 
\[
\delta_\lambda(x,y) := (\lambda x,\lambda^2 y) .
\]
In this case, $\R^2 = V_1\oplus V_2$ with $V_1=\R\times\{0\}$ and $V_2 = \{0\}\times\R$, and $\bar\delta(x,y)=(x,2y)$.
Condition \eqref{eq04221835} holds for all $(x,y)\in\R^2$ with $y\neq0$.
One can actually show that, for any homogeneous metric on $(\R^2,\delta_\lambda)$ with closed unit ball $B$ centered at $0$, the set $I=\{x\in\R:(x,0)\in B\}$ is a closed interval and
there exists a function $f:I\to\R$ that is locally Lipschitz on the interior of $I$
%
such that
\[
\bb S_d\cap\{(x,y):y\ge0\} = \{(x,f(x)): x\in I\}. 
\]
We will prove a similar statement in the Heisenberg group with an argument that applies here too, see Section~\ref{sec05261301}.

The third example is given by the dilations
\begin{equation}\label{eq04210906}
\delta_\lambda(x,y) := (\lambda^2 x,\lambda^2 y) ,
\end{equation}
and it is interesting because of the next proposition.
\begin{Prop}\label{prop05281759}
 	There exists a homogeneous (with respect to dilations \eqref{eq04210906}) distance $d$ on $\R^2$ whose unit sphere has Euclidean Hausdorff dimension  $\frac32$.
\end{Prop}
Notice that $\frac32$ is the maximal Hausdorff dimension that one gets by the estimate \eqref{eq04282051}.

For proving Proposition~\ref{prop05281759}, 
we need to find a set $B\subset\R^2$ that satisfies all four conditions
listed in Lemma~\ref{lem05282259}, in particular
\begin{equation}\label{eq05281112}
 	\forall p,q\in B,\ \forall t\in[0,1]\quad t^2p + (1-t)^2q \in B .
\end{equation} 

One easily proves the following preliminary facts.
\begin{Lem}\label{lem05281521}
	Let $p,q\in\R^2$ and $\gamma:[0,1]\to\R^2$, $\gamma(t) := t^2p + (1-t)^2q$.
	\begin{enumerate}
	\item 	The curve $\gamma$ is contained in the triangle of vertices $0,p,q$.
	\item 	The curve $\gamma$ is an arc of the parabola passing through $p$ and $q$ and that is tangent to the lines $\Span\{p\}$ and $\Span\{q\}$.
	\item 	If $B$ satisfies \eqref{eq05281112} and $A:\R^2\to\R^2$ is a linear map, then $A(B)$ satisfies \eqref{eq05281112} as well.
	\end{enumerate}
\end{Lem}
%
\begin{Lem}\label{lem05281120}
 	For $0<C\le1$, define
	\[
	Y_C :=\{(x,y):|x|\le1,\ y\le 1+C\sqrt{|x|}\}.
	\]
	Then $Y_C$ satisfies \eqref{eq05281112}.
\end{Lem}
\begin{proof}
 	Let $p,q\in Y_C$ and set $\gamma(t)= (\gamma_x(t),\gamma_y(t)):=t^2 p+ (1-t)^2q$.
	
	If both $p$ and $q$ stay on one side with respect to the vertical axis, then $\gamma(t)\in Y_C$ for all $t\in[0,1]$ thanks to the first point of Lemma~\ref{lem05281521} and because the two sets $Y_C\cap\{x\ge0\}$ and $Y_C\cap\{x\le0\}$ are convex.
		
	Therefore, we suppose that
	\[
	p = (-p_x,p_y) \qquad q=(q_x,q_y)
	\]
	with $p_x,q_x>0$.
	Let $t_0\in[0,1]$ be the unique value such that $\gamma_x(t_0)=0$.
	Then the curve $\gamma$ lies in the union of the two triangles with vertices $0,\gamma(0),\gamma(t_0)$ and $0,\gamma(1),\gamma(t_0)$, respectively.
	Therefore, $\gamma$ lies in $Y_C$ if and only if $\gamma_y(t_0)\le 1$.
	Solving the equation $\gamma_x(t_0)=t^2_0 (q_x-p_x) - 2q_x t_0 + q_x=0$, one gets
	\[
	t_0 = \frac{ \sqrt{q_x} }{ \sqrt{q_x} + \sqrt{p_x} },
	\qquad
	(1-t_0) = \frac{ \sqrt{p_x} }{ \sqrt{q_x} + \sqrt{p_x} } .
	\]
	From the expression of $\gamma_y(t_0) = t_0^2 p_y + (1-t_0)^2 q_y$,
	we notice that, $p_x$ and $q_x$ fixed, the worst case is when $p_y$ and $q_y$ are maximal, i.e.,
	\[
	p_y = 1+C\sqrt{p_x}, \qquad q_y = 1+C\sqrt{q_x} .
	\]
	
	Finally
	\begin{align*}
	\gamma_y(t_0) 
	&= t_0^2 p_y + (1-t_0)^2 q_y  \\
	&= \frac{ q_x }{ (\sqrt{q_x} + \sqrt{p_x})^2 } ( 1+C\sqrt{p_x} ) +  \frac{ p_x }{ (\sqrt{q_x} + \sqrt{p_x})^2 } ( 1+C\sqrt{q_x} )  \\
	&= \frac{1}{ (\sqrt{q_x} + \sqrt{p_x})^2 } \left( q_x+p_x+ Cq_x\sqrt{p_x} + Cp_x\sqrt{q_x} \right)  \\
	&= 1 + \frac{ -2\sqrt{p_xq_x} + Cq_x\sqrt{p_x} + Cp_x\sqrt{q_x} }{ (\sqrt{q_x} + \sqrt{p_x})^2 }  \\
	&= 1 + \sqrt{p_xq_x} \frac{ -2 + C(\sqrt{q_x} + \sqrt{p_x}) }{ (\sqrt{q_x} + \sqrt{p_x})^2 } .
	\end{align*}
	Since $-2 + C(\sqrt{q_x} + \sqrt{p_x})\le 0$, then we have $\gamma_y(t_0) \le 1$, as wanted.
\end{proof}
\begin{Lem}\label{lem05281752}
 	Let $\alpha,\beta>0$. 
	For all $0<\epsilon\le \alpha$ and all $0< C\le\frac{\beta}{\sqrt\alpha}$, the set
	\[
	Y(\epsilon,\beta,C) := \{(x,y): |x|\le \epsilon,\ y\le \beta + C \sqrt{|x|}
	\]
	satisfies \eqref{eq05281112}.
\end{Lem}
\begin{proof}
	Define the linear map $A(x,y):=(\alpha x,\beta y)$ and set $C':=C\frac{\sqrt\alpha}{\beta} \le 1$.
	Then one just needs to check that
	\[
	Y(\epsilon,\beta,C) = A(Y_{C'}) \cap \{(x,y):|x|\le \epsilon\} ,
	\]
	where $Y_{C'}$ is defined as in the previous Lemma~\ref{lem05281120}.
\end{proof}
\begin{proof}[Proof of Proposition~\ref{prop05281759}]
	First of all, let $\theta_0>0$ be such that for all $|\theta|\le\theta_0$ it holds
	\begin{equation}
	\frac{|\theta|}{2} \le |\cos(\frac\pi2+\theta)| = |\sin\theta| \le 2 |\theta| .
	\end{equation}
	Moreover, let $L,m,M,C>0$ be such that
	\[
	\frac{L\sqrt{2}}{\sqrt{m}} \le
	C
	\le \frac{m}{\sqrt{2M\theta_0}} .
	\]
	
	Let $f:\R\to(0,+\infty)$ be a function such that
	\begin{equation}\label{eq05282238}
	 	\forall s,t\in\R\quad |f(t)-f(s)| \le L \sqrt{|t-s|} ,
	\end{equation}
	\begin{equation}\label{eq05282239}
	 	\forall t\in\R\quad m\le f(t)\le M .
	\end{equation}
	
	We claim that, for $|\theta| \le \theta_0$, we have
	\begin{equation}\label{eq05281847}
	f\left(\frac\pi2+\theta\right)\cdot \left(\cos(\frac\pi2+\theta),\sin(\frac\pi2+\theta) \right)
	\in
	Y\left(2M\theta_0, f(\frac\pi2), C \right)
	\end{equation}
	where $Y\left(2M\theta_0, f(\frac\pi2), C \right)$ is defined as in Lemma~\ref{lem05281752}.
 	Indeed, we have on one side
	\[
	|x| :=
	|f(\frac\pi2+\theta) \cos(\frac\pi2+\theta)|
	\le M 2|\theta| \le 2M\theta_0.
	\]
	On the other side,
	\begin{align*}
	y &:=
	f(\frac\pi2+\theta) \sin(\frac\pi2+\theta) 
	\le f(\frac\pi2+\theta)  \\
	&\le f(\frac\pi2) + f(\frac\pi2+\theta) - f(\frac\pi2) 
	\le f(\frac\pi2) + L \sqrt{|\theta|}  \\
	&\le f(\frac\pi2) + L \frac{ \sqrt{2\cos(\frac\pi2+\theta)f(\frac\pi2+\theta) } }{\sqrt{f(\frac\pi2+\theta)}}
	\le f(\frac\pi2) + \frac{\sqrt2 L}{\sqrt{m}} \sqrt{|x|} \\
	&\le f(\frac\pi2) + C \sqrt{|x|} .
	\end{align*}
	So \eqref{eq05281847} is satisfied.
	
	Since for $\alpha :=2M\theta_0$ and $\beta:=f(\frac\pi2)$ we have
	\[
	\frac{\beta}{\sqrt{\alpha}} 
	= \frac{f(\frac\pi2)}{\sqrt{ 2M\theta_0}}
	\ge \frac{ m }{\sqrt{ 2M\theta_0}}
	\ge C,
	\]
	Lemma \ref{lem05281752} applies and we get that $Y\left(2M\theta_0, f(\frac\pi2), C \right)$ satisfies \eqref{eq05281112}.
	
	For any $\theta$ we set $A_\theta$ to be the counterclockwise rotation of angle $\theta$:
	\[
	A_\theta = 
	\begin{pmatrix}
	\cos\theta & -\sin\theta \\
	\sin\theta & \cos\theta
	\end{pmatrix} .
	\]
	Define the curve $\phi(t) := f(t) (\cos t,\sin t )$.
	Notice that $A_\theta\phi(t) = f((t-\theta)+\theta) (\cos (t+\theta),\sin (t+\theta) )$ and that the function $s\mapsto f(s+\theta)$ is still satisfying both \eqref{eq05282238} and \eqref{eq05282239}.
	So we have that,
	for $|t|,|s|< \frac{\theta_0}2$, 
	\[
	\phi(\frac\pi2+t) \in A_s[ Y(2M\theta_0 ,f(\frac\pi2+s),C) ] 
	\]
	and the set $A_s[ Y(2M\theta_0 ,f(\frac\pi2+s),C) ]$ satisfies \eqref{eq05281112}.
	
	Set
	\[
	B := \bigcap_{|s|< \frac{\theta_0}2 } 
	\left(A_s[ Y(2M\theta_0 ,f(\frac\pi2+s),C) ] \cap -A_s[ Y(2M\theta_0 ,f(\frac\pi2+s),C) ] \right) .
	\]
	The set $B$ satisfies all three conditions of Lemma~\ref{lem05282259}, hence it is the unit ball of a homogeneous metric.
	Moreover,
	\[
	\{\phi(\frac\pi2+t): |t|< \frac{\theta_0}2\}
	\subset \de B.
	\]
	
	The statement of Proposition~\ref{prop05281759} follows because there are functions $f:\R\to[0,+\infty)$ that satisfy \eqref{eq05282238} and \eqref{eq05282239} and such that the image of the curve $\phi$ has Hausdorff dimension $\frac32$.
	Indeed, the image of $\phi$ has the same Hausdorff dimension of the graph of $f$, and then one uses the sharpness of Lemma~\ref{lem04251012}.
\end{proof}

\begin{Rem}\label{rem09051840}
	Using the same arguments as in the proof of Lemma~\ref{lem05281120}, one easily shows that the set
	\[
	B := \{(x,y)\in\R^2:\ |x|\le 1,\ -f(-x) \le y \le f(x) \},
	\]
	where
	\[
	f(x) :=
	\begin{cases}
		1 & x\le 0 \\
	 	1+\sqrt{x} & x>0 ,
	\end{cases}
	\]
	is the unit ball of a homogeneous distance on $\R^2$ with dilations \eqref{eq04210906}.
	Notice that such $B$ is a Lipschitz domain, but the associated homogeneous distance is not Lipschitz in any neighborhood of the point $(0,1)$, thanks to Lemma~\ref{lem09051217}.
\end{Rem}

\subsection{The Heisenberg groups}
In the Heisenberg groups $\HH^n$ (for an introduction see \cite{MR2312336}) 
condition \eqref{eq04221835} holds at every non-zero point.
Therefore, balls of any homogeneous metric on $\HH^n$ are Lipschitz domains.
We will treat more in detail the first Heisenberg group in Section~\ref{sec05261301}.

\subsection{A sub-Finsler sphere with a cusp}
Let $\HH$ be the first Heisenberg group (see Section~\ref{sec05261301} for the definition).
The group $G=\HH\times\R$ is a stratified group with grading $(V_1\times\R)\oplus V_2$, where $V_1\oplus V_2$ is a stratification for $\HH$.
The line $\{0_\HH\}\times\R$ is a singular curve in $G$.
Moreover, it has been shown in \cite{MR3153949} that there exists a sub-Finsler distance on $G$ whose unit sphere $\bb S_d$ has a cusp in the intersection $\bb S_d\cap (\{0_\HH\}\times\R)$.
However, for sub-Riemannian metrics we still have balls that are Lipschitz domains, as the following Proposition~\ref{prop07122316} shows.
But let us first recall a simple fact:
\begin{Lem}\label{lem07122315}
	Let $\bb A$ and $\bb B$ be two stratified groups with stratifications $\bigoplus V_i$ and $\bigoplus W_i$, respectively.
	Endow $V_1$ and $W_1$ with a scalar product each and let $d_A,d_B$ be the corresponding homogeneous sub-Riemannian distances.
	
	Then $\bb A\times\bb B$ is a Carnot group with stratification $\bigoplus_i V_i\times W_i$ and metric 
	\[
	d((a,b),(a',b')) := \sqrt{ d_A(a,a')^2 + d_B(b,b')^2 },
	\]
	which is the sub-Riemannian metric generated by the scalar product on $V_1\times W_1$ 
	induced by the scalar products on $V_1$ and $W_1$. 
\end{Lem}
One proves this lemma by using the fact that the energy of curves on $\bb A\times\bb B$ (i.e., the integral of the squared norm of the derivative) is the sum of the energies of the two components of the curve.
\begin{Prop}\label{prop07122316}
	Any homogeneous sub-Riemannian metric on $\HH\times\R$ is locally Lipschitz away from the diagonal.
\end{Prop}
\begin{proof}
	First of all, we show that, up to isometry, there is only one homogeneous sub-Riemannian distance on $\HH\times\R$.
	Let $(X_1,Y_1,T_1)$ and $(X_2,Y_2,T_2)$ be two basis of $V_1\times\R$ that are orthonormal for two sub-Riemannian structures, respectively.
	We may assume $T_1, T_2\in\{0\}\times\R$.
	Notice that $[X_i,Y_i]\notin V_1\times\R$.
	The linear map such that $X_1\mapsto X_2$, $Y_1\mapsto Y_2$, $T_1\mapsto T_2$, $[X_1,Y_1]\mapsto[X_2,Y_2]$ is an automorphism of Lie algebras and induces an isometry between the two sub-Riemannian structures.
	
	Denoting by $d_\HH$ and $d_\R$ the standard metrics on $\HH$ and $\R$, respectively, we prove the proposition for the product metric as in Lemma~\ref{lem07122315}.
	Namely, we need to check that %
%
%
%
	the function
	\begin{equation}\label{eq05261620}
	(p,t)\mapsto d((0,0),(p,t)) = \sqrt{d_\HH(0,p)^2 + t^2}
	\end{equation}
	is locally Lipschitz at all $(\hat p,\hat t)\neq(0,0)$.
	This follows directly from Proposition~\ref{prop06031443}.
\end{proof}

\subsection{A sub-Riemannian sphere with a cusp}
\begin{Prop}
 	Let $G$ be a Carnot group of step 3 endowed with a sub-Riemannian distance $d_G$.
	Then the sub-Riemannian distance $d$ on $G\times\R$ given by
	\[
	d((p,s),(q,t)) = \sqrt{ d_E(p,q)^2 + |t-s|^2 }
	\]
	has a unit sphere with a cusp at $(0_G,1)$.
\end{Prop}
\begin{proof}
	Let $\g=\bigoplus_{i=1}^3V_i$ be the Lie algebra of $G$ and fix $Z\in V_3\setminus\{0\}$.
	We identify $\g$ with $G$ through the exponential map.
	
	The intersection of the unit sphere in $(G\times\R,d)$ with the plane $\Span\{Z\}\times\R$ is given by all points $(zZ,t)$ such that
	\begin{equation}\label{eq05271758}
	d_G(0,zZ)^2 + t^2 = 1 .
	\end{equation}
	Since $d_G$ is homogeneous on $G$, there exists $C>0$ such that for all $z\in\R$
	\begin{equation}\label{eq05271759}
	d_G(0,zZ) = C |z|^{\frac13} .
	\end{equation}
	Putting together \eqref{eq05271758} and \eqref{eq05271759} we obtain that this intersection consists of all the points $(zZ,t)$ such that
	\[
	|z| = \left( \frac{1+t}{C^2} \right)^{\frac 32} \cdot (1-t)^{\frac 32} .
	\]
	One then easily sees that this set in $\R^2$ has a cusp at $(0,1)$.
\end{proof}

\section{A closer look at the Heisenberg group}\label{sec05261301}
The Heisenberg group $\HH$ is the easiest example of a stratified group that is not Abelian and for this reason it has been studied in large extend.
The most common homogeneous metrics on $\HH$ are the Koràny metric and the sub-Riemannian metric.
Sub-Finsler metrics on $\HH$ arise in the study of finitely-generated groups, see \cite{MR3267520} and references therein.
The geometry of spheres has been studied in \cite{2014arXiv1406.1484L} and \cite{MR3254527}.\\

The Lie algebra $\h$ of the Heisenberg group is a three dimensional vector space $\Span\{X,Y,Z\}$ with a Lie bracket operation defined by the only nontrivial relation $[X,Y] = Z$.

We identify the Heisenberg group $\HH$ again with $\Span\{X,Y,Z\}$, where we define the group operation
\[
p \cdot q := p+q+\frac12[p,q]
\qquad\forall p,q\in\HH .
\]
Hence $\h$ is the Lie algebra of $\HH$ and the exponential map $\h\to\HH$ is the identity map.
Notice that the inverse of an element $p$ is $p^{-1}=-p$.

The Heisenberg algebra admits the stratification $\h = V_1\oplus V_2$ with $V_1=\Span\{X,Y\}$ and $V_2=\Span\{Z\}$.
Denote by $\pi$ the linear projection $\h\to V_1$ along $V_2$.
Notice that this map, regarded as $\pi:(\HH,\cdot)\to (V_1,+)$, is a group morphism.

The dilations $\delta_\lambda:\HH\to\HH$ are explicitly expressed by
\begin{equation*}
\delta_\lambda(xX+yY+zZ) = x\lambda X+y\lambda Y+z\lambda^2 Z ,
\qquad\forall \lambda>0 .
\end{equation*}
These are both Lie algebra automorphisms $\delta_\lambda:\h\to\h$ and Lie group automorphisms $\HH\to\HH$.

Three are the main results of this section.

\begin{Prop}\label{prop05291014}
 	Let $N:\HH\to[0,+\infty)$ be a homogeneous norm.
	Then the unit ball
	\[
	B := \{p\in\HH:N(p)\le 1\}
	\]
	is a star-like Lipschitz domain.
\end{Prop}
\begin{proof}
	One easily shows that condition \eqref{eq04221835} holds for all $p\in\HH\setminus\{0\}$.
	In order to prove that $B$ is star-like, one first notice that if $p\in B$, then $-p\in B$, hence 
	$
	\delta_t(p)\delta_{1-t}(-p) = (2t-1) p
	\in B
	$ 
	for all $t\in[0,1]$, and this is a straight line passing through zero.
\end{proof}
\begin{Prop}\label{prop05291027}
 	Let $N$ and $B$ as in Proposition~\ref{prop05291014}.
	Set $K:=\pi(B)\subset V_1$.
	Then $K$ is a compact, convex set with $K=-K$ and $K=cl(\Int(K))$, and there 
	exists a function $f:K\to[0,+\infty)$, locally Lipschitz on $\Int(K)$, such that
	\begin{equation}\label{eq08291052}
	B = \{v+zZ:\ v\in K,\ -f(-v)\le z\le f(v) \} .
	\end{equation}
\end{Prop}
The proof is postponed to Section~\ref{subs07131005}.

We remark that homogeneous distances and sub-Finsler homogeneous distances on $\HH$ have a precise relation.
Indeed, if $d$ is a homogeneous distance on $\HH$, then it is easy to show that the length distance generated by $d$ is exactly the sub-Finsler distance that has the norm on $V_1$ generated by the set $K$ defined in Proposition~\ref{prop05291027}.
\begin{Prop}\label{prop05291028}
	Let $K\subset V_1$ be a compact, convex set with $-K=K$ and $0\in\Int(K)$.
	Let $g:K\to\R$ be Lipschitz.
	Then there exists $b\in\R$ such that for $f:=g+b$ the set $B$ as in \eqref{eq08291052}
	is the unit ball of a homogeneous norm.
\end{Prop}
The proof will appear in Section~\ref{subs07131006}.

As a consequence of Proposition~\ref{prop05291028}, we get the existence of homogeneous distances on $\HH$ that are not almost convex in the sense of \cite{2014arXiv1411.4201D}.
Indeed, one can take the distance associated to $g(xX+yY)=|x|$ from Proposition~\ref{prop05291028}.
\label{page09081203}

\subsection{Proof of Proposition~\ref{prop05291027}}\label{subs07131005}

\begin{Lem}\label{lem1630}
	Let $B\subset\HH$ be an arbitrary closed set satisfying \eqref{eq05291157}.	
 	If $p = v+zZ\in B$ with $v=\pi(p)\in V_1$, then $v+szZ\in B$, for all $s\in[0,1]$.
	In particular,
	\begin{enumerate}
	\item 	$\pi(B) = B\cap V_1$;
	\item 	$\pi(B)\subset V_1$ is convex.
	\end{enumerate}
\end{Lem}
\begin{proof}
 	We have that for all $t\in[0,1]$
	\[
	B \ni \delta_tp\cdot  \delta_{1-t}p 
	= v + (t^2+ (1-t)^2)zZ .
	\]
	Since the image of $[0,1]$ through the map $t\mapsto (t^2+ (1-t)^2)$ is $[\frac12,1]$, then it follows $v+szZ\in B$ for all $s\in[\frac12,1]$.	
	Iterating this process and using the closeness of $B$, we get $v+szZ\in B$ for all $s\in[0,1]$.	
	For the last statement, take $v,w\in\pi(B)\subset B$ and notice that $tv+(1-t)w = \pi(\delta_tv\cdot\delta_{1-t}w) \in \pi(B)$.
\end{proof}

Let $B = \{N\le 1\}$ be the unit ball of a homogeneous norm and set $K:=\pi(B)\subset V_1$ and
$\Omega:=\Int(K)$.
First, we check that $\bar\Omega=K$. 
On the one hand, clearly we have $\bar\Omega\subset K$.
On the other hand, if $v\in K$, then for any $t\in[0,1)$ we have $N(\delta_tv)=tN(v)<1$, i.e., $\delta_tv=tv\in\Int B\cap V\subset\Omega$. Hence $v\in\bar\Omega$.

If we define $f:K\to [0,+\infty)$ as $f(v):=\max\{z:v+zZ\in Q\}$, then we have \eqref{eq08291052}.
In order to prove that $f$ is locally Lipschitz on $\Omega$, we need to prove
\begin{equation}\label{eq05291134}
\begin{array}{c}
 	\forall p\in\de B\cap\{z\ge 0\}\cap\pi^{-1}(\Omega),\\
	\exists U\ni p\text{ open }, \exists C\text{ vertical cone, s.t.}\\ 
	\forall q\in U\cap\de B
	\text{ it holds }q+C\subset B.
\end{array}
\end{equation}
Here a \emph{vertical cone} is a Euclidean cone with axis $-Z$ and non-empty interior.

So, fix $p\in\de Q\cap\{z\ge 0\}$ such that $\pi(p)\in\Omega$.
Define for $\theta\in\R$ and $\epsilon>0$
\[
v_\theta := x_\theta X + y_\theta Y := \epsilon\cos(\theta) X + \epsilon\sin(\theta)Y .
\]
For $\epsilon>0$ small enough, $\pi(p)+v_\theta\in\Omega$ for all $\theta$.
Define
\[
\phi(t,\theta) := \delta_{(1-t)}p\cdot\delta_t(\pi(p)+v_\theta) .
\]
Clearly $\phi(t,\theta)\in B$ for $t\in[0,1]$ and $\theta\in\R$, and $\phi(0,\theta)=p$ for all $\theta$.
Geometrically, $\phi([0,1]\times\R)$ is a ``tent'' inside $B$ standing above the whole vertical segment from $\pi(p)$ to $p$.
Notice that $p\neq\pi(p)$, 
indeed $N(p)=1$ while $N(\pi(p))<1$, because $\pi(p)\in\Omega$.

We only need to prove that the curves $t\mapsto \phi(t,\theta)$ meet this vertical segment by an angle bounded away from 0.
Some computations are needed: set $p=p_1X+p_2Y+p_3Z$, then
\[
\phi(t,\theta) = \pi(p)  + t v_\theta + \left(\frac12 t(1-t) (p_1y_\theta - p_2x_\theta) + (1-t)^2 p_3 \right) Z .
\]
We take care only of the third coordinate. 
Set
\begin{align*}
g(t) :&= \frac12 t(1-t) (p_1y_\theta - p_2x_\theta) + (1-t)^2 p_3  \\
&= t^2\ (-\frac12(p_1y_\theta - p_2x_\theta) + p_3)\ +\ t\ (\frac12(p_1y_\theta - p_2x_\theta) -2p_3)\ +\ p_3 .
\end{align*}
Saying that the angle between the curve $t\mapsto \phi(t,\theta)$ and the vertical segment at $p$ is uniformly grater than zero, is equivalent to give an upper bound to the derivative of $g$ at $0$ for all $\theta$.
Since
\[
g'(0) = \frac12(p_1y_\theta - p_2x_\theta)-2p_3 ,
\]
we are done.

Finally, since both $\epsilon$ and $g'(0)$ depend continuously on $p$, then \eqref{eq05291134} is satisfied.
\qed

\subsection{Proof of Proposition~\ref{prop05291028}}\label{subs07131006}

We consider the bilinear map $\omega:V_1\times V_1\to\R$ given by
\[
\omega(v_1X+v_2Y , w_1X+w_2Y) := v_1w_2-v_2w_1 .
\]


\begin{Lem}\label{lem08291227}
	For any continuous function $f:K\to[0,+\infty)$, 
 	the set $B$ as in \eqref{eq08291052}
	is the unit ball of a homogeneous norm on $\HH$
	if and only if
	\begin{equation}\label{eq05291146}
	\begin{array}{c}
		\forall v,w\in K\quad\forall t\in[0,1]\\
	 	f(tv+(1-t)w) - t^2 f(v) - (1-t)^2f(w) - \frac{t(1-t)}{2}\omega(v,w)\ge 0 .
	\end{array}
	\end{equation}
\end{Lem}
\begin{proof}
	One easily sees that $B=B^{-1}$.
	Notice that $B$ is the unit ball of a homogeneous norm if and only if it satisfies \eqref{eq05291157}.
	
	\ddx
	Assume that $B$ satisfies \eqref{eq05291157}.
	Then for any $v,w\in K$ we have
	\begin{multline*}
	B\ni \delta_t(v+f(v)Z)\cdot\delta_{(1-t)}(w+f(w)Z) = \\
	= t v+(1-t)w +\left(t^2 f(v) +(1-t)^2 f(w) + \frac12 t(1-t)\omega(v,w)\right)Z ,
	\end{multline*}
	hence
	\[
	t^2 f(v) +(1-t)^2 f(w) + \frac12 t(1-t)\omega(v,w) \le f(t v+(1-t)w) .
	\]

	\ssx
	Suppose $f$ satisfies \eqref{eq05291146}.
	Define
	\begin{align*}
	 	B^+ &:=  \{v+zZ: v\in K\text{ and } z \le f(v)\}, \\
		B^- &:=  \{v+zZ: v\in K\text{ and }-f(-v)\le z \} .
	\end{align*}
	We will show that both $B^+$ and $B^-$ satisfy \eqref{eq05291157}, from which it follows that $B=B^+\cap B^-$ satisfies \eqref{eq05291157} as well.
	
	So, let $v,w\in K$ and $z_1,z_2\in\R$ such that $v+z_1Z,w+z_2Z\in B^+$. 
	Then the third coordinate of $\delta_t(v+z_1Z)\cdot\delta_{(1-t)}(w+z_2Z)$ satisfies
	\begin{multline*}
	t^2z_1+(1-t)^2z_2 + \frac12 t(1-t)\omega(v,w) \le \\
	\le t^2f(v)+(1-t)^2f(w) + \frac12 t(1-t)\omega(v,w)
	\le f(t w+(1-t)v),
	\end{multline*}
	therefore we have $\delta_t(v+z_1Z)\cdot\delta_{(1-t)}(w+z_2Z)\in B^+$ for all $t\in[0,1]$.

	The calculation for $B^-$ is similar.
\end{proof}

The verification of the next lemma is simple and therefore it is omitted.
\begin{Lem}\label{lem08291235}
 	Suppose that $g: K\to\R$ is a continuous function such that there is a constant $A\in\R$ with
	\begin{equation}\label{eq1158}
	\begin{array}{c}
	\forall v,w\in K,\ \forall t\in[0,1] \\
	g(tv+(1-t)w) - t^2 g(v) - (1-t)^2g(w) \ge  A t(1-t) .
	\end{array}
	\end{equation}
	Then $f:=g+B$ satisfies \eqref{eq05291146} with
	\[
	B := \sup_{v,w\in K} \frac12\left(\frac12\omega(v,w)-A\right)
	=\frac14\left(\sup_{v,w\in K}\omega(v,w)\right)- \frac12A .
	\]
\end{Lem}
\begin{Lem}\label{lem08291236}
 	Let $g:K\to\R$ be $L$-Lipschitz. Then $g$ satisfies \eqref{eq1158} for
	\[
	A:=-2L\diam(K) - 4\sup_{p\in K} |g(p)| .
	\]
\end{Lem}
\begin{proof}
 	Notice that we need to show that \eqref{eq1158} holds only for $t\in(0,1)$ and that \eqref{eq1158} is symmetric in $t$ and $(1-t)$.
	So, it is enough to consider only the case $t\in(0,\frac12]$:
	\begin{multline*}
	 	\frac{ g(tv+(1-t)w) - t^2 g(v) - (1-t)^2g(w) }{ t(1-t) } \\
		= \frac{ g(w+t(v-w)) - g(w) }{ t(1-t) } + \frac{ (1-(1-t)^2) g(w) }{ t(1-t) } - \frac{t}{1-t} g(v) \\ 
		\ge - \frac{L\|v-w\|}{1-t} + \frac{2-t}{1-t} g(w) - \frac{t}{1-t} g(v) \\ 
		\ge -2L\diam(K) - 4\sup_{p\in K} |g(p)| .
	\end{multline*}
\end{proof}
Putting together Lemmas \ref{lem08291236}, \ref{lem08291235}, and \ref{lem08291227}, we get Proposition~\ref{prop05291028}.
\qed


\printbibliography
\end{document}